\documentclass[onecolumn]{IEEEtran}

\usepackage{amsmath}   
\usepackage{amssymb}

\RequirePackage{amsthm,amsmath,amssymb,amsfonts,graphicx, subfigure} 
\RequirePackage[colorlinks]{hyperref}

\newtheorem{thm}{Theorem}
\newtheorem{question}{Question}
\newtheorem{prop}{Proposition}
\newtheorem{lem}{Lemma}
\newtheorem{defn}{Definition}
\newtheorem{rmk}{Remark}

\newcommand{\bc}{\begin{center}}
\newcommand{\ec}{\end{center}}
\newcommand{\bt}{\begin{tabular}}
\newcommand{\et}{\end{tabular}} 
\newcommand{\bea}{\begin{eqnarray}}
\newcommand{\eea}{\end{eqnarray}}

\newcommand{\ba}{\begin{array}}
\newcommand{\ea}{\end{array}}

\def\be{\begin{eqnarray}}
\def\ee{\end{eqnarray}}
\def\ben{\begin{eqnarray*}}
\def\een{\end{eqnarray*}}

\newcommand{\ra} {\rightarrow}

\newcommand{\RL}{{\mathbb R}}

\newcommand{\calL}{\mbox{${\cal L}$}}

\newcommand{\Rpl}{{\mathbb R}_{+}}
\newcommand{\Zpl}{\mathbb{Z}_{+}}
\newcommand{\Nat}{\mathbb{N}}

\def\sq{$\Box$}

\def\qed{\ifmmode\sq\else{\unskip\nobreak\hfil
\penalty50\hskip1em\null\nobreak\hfil\sq
\parfillskip=0pt\finalhyphendemerits=0\endgraf}\fi\par\medbreak}

\def\tr{{\rm tr\, }}

\newsavebox{\junk}
\savebox{\junk}[1.6mm]{\hbox{$|\!|\!|$}}

\def\det{{\mathop{\rm det}}}

\newcommand{\one}{\hbox{\rm\textbf{1}}}

\def\half{{\mathchoice{\textstyle \frac{1}{2}}%
{\frac{1}{2}}%
{\hbox{\tiny $\frac{1}{2}$}}%
{\hbox{\tiny $\frac{1}{2}$}} }}

 \def\eq#1/{(\ref{#1})}

\def\eq#1/{(\ref{e:#1})}

\newcommand{\setS}{s}
\newcommand{\setT}{t}
\newcommand{\collS}{\mathcal{G}}

\newcommand{\rth}{\frac{1}{r}}
\newcommand{\sumS}{\sum_{\setS\in\collS}}

\newcommand{\bs}{\beta_{\setS}}

\newcommand{\Ys}{Y_{\setS}}
\newcommand{\Ex}{\text{Ex}}

\newcommand{\calN}{\mathcal{N}}
\newcommand{\calB}{\mathcal{B}}
\newcommand{\calO}{\mathcal{O}}
\newcommand{\nullset}{\phi}
\newcommand{\lam}{\lambda}
\newcommand{\vol}{\text{Vol}}

\def\R{{\bf R}}

\def\bee{\begin{eqnarray*}}
\def\ene{\end{eqnarray*}}

\begin{document}
\title{Combinatorial Entropy Power Inequalities: A Preliminary Study of the Stam region}
\author{Mokshay~Madiman,~\IEEEmembership{Senior Member,~IEEE}, 
and Farhad Ghassemi%
\thanks{M. Madiman is with the Department of Mathematical Sciences, University of Delaware, USA. 
F.~Ghassemi is with the Sloan School of Management, MIT, Cambridge, MA 02142, USA.
Email for correspondence: {\tt madiman@udel.edu}}
\thanks{M. Madiman is grateful for support from the U.S. National Science Foundation 
through the grants DMS-1409504 (CAREER) and CCF-1346564.
}
\thanks{A preliminary version of this work \cite{MG09:isit} was presented at the 
2009 IEEE International Symposium on Information Theory in Seoul, Korea.
The conference paper had correctly stated results, but discussed one of them in an erroneous manner
in the text; the error is corrected here.}
}
\markboth{Submitted to IEEE Transactions on Information Theory, 2018}{Madiman, Ghassemi}

\maketitle

\begin{abstract}
We initiate the study of the Stam region, defined as the subset of the positive orthant 
in $\mathbb{R}^{2^n-1}$ that arises from considering entropy powers
of subset sums of $n$ independent random vectors in a Euclidean space of finite dimension. 
We show that the class of fractionally superadditive set functions provides
an outer bound to the Stam region, resolving a conjecture
of  A.~R.~Barron and the first author. 
On the other hand, the entropy power of a sum of independent random vectors
is not supermodular in any dimension. We also develop some qualitative properties
of the Stam region, showing for instance that its closure is a logarithmically convex cone.
\end{abstract}

\section{Introduction}
\label{sec:intro}

For a $\RL^{d}$-valued random vector $X$ with density $f$
with respect to the Lebesgue measure  on $\RL^d$,
the differential entropy is 
\ben
h(X)=-\int f(x) \log f(x) dx ,
\een
if it exists. When $X$ is supported on a strictly lower-dimensional
set than $\RL^d$ (and hence does not have a density with respect to the
$d$-dimensional Lebesgue measure), a limiting argument suggests
that one should set $h(X) =-\infty$. 
Also, by considering for example the distribution supported on
$(e,\infty)\subset\RL$ with 
density
\ben
f(x)=\frac{1}{x(\log x)^2} ,
\een
it is seen that the entropy can take the value $h(X) =+\infty$.

Unless explicitly mentioned, we limit ourselves to random vectors $X$
with $h(X)<+\infty$.
The entropy power of $X$ is defined by
\ben
\calN(X)=e^{2h(X)/d}. 
\een
Then, $\calN(X)\in \Rpl:=[0,\infty)$ is a non-negative
real number. 
If a random vector $X$ has a well-defined, finite covariance matrix $K$
(i.e., the variance of each of the $d$ components is finite),
then it is well known that 
\ben
h(X)\leq h(N(0,K)) = \half\log [(2\pi e)^d \det(K) ] < +\infty ,
\een
because of the fact that Gaussians maximize entropy under a
constraint on the covariance matrix. Thus the class of random vectors
under consideration certainly includes all those 
whose components have finite variances,
but in fact is much larger since it also includes many heavy-tailed distributions.

There are two main motivations for considering entropy power
inequalities: the first comes from the fact that it is related to probabilistic
isoperimetric phenomena including the entropic central limit theorem
(see, e.g., Barron \cite{Bar86}), and the second comes from the fact that
it can be extremely useful in the study of rate and capacity regions
in multi-user information theory (see, e.g., Shannon \cite{Sha48},
Bergmans \cite{Ber74}, Ozarow \cite{Oza80}, Costa \cite{Cos85a}
and Oohama \cite{Ooh98}).

Let  $X_1, X_2, \ldots,X_{n}$ be independent
random vectors. 
We write $[n]$ for the index set $\{1,2,\ldots,n\}$,
and $\phi$ for the empty set.  
For any nonempty $\setS\subset [n]$, define the subset sum
\ben
\Ys=\sum_{i\in\setS} X_{i} .
\een
One is interested in the entropy powers $\calN(\Ys)$
of the subset sums, which leads naturally to the following
objects of study:
\be\begin{split}
\Gamma_d(n)&=\{ [\calN(\Ys)]_{\setS\subset[n]} : X_1, X_2, \ldots,X_{n} \\
& \quad\quad\text{are  independent $\RL^{d}$-valued random vectors}\\
& \quad\quad\text{with $h(X_1+\ldots+X_n)<\infty$}\} .
\end{split}\ee
We now give a more precise definition of the relevant objects.

\begin{defn}\label{defn:stam}
Let ${\bf F}_{d,n}$ be the collection of all $n$-tuples  ${\bf f}=(f_1, \ldots, f_n)$
of  probability density functions $f_i$ on $\RL^d$,
such that if $X_i\sim f_i$ are independent, their sum has finite entropy
(i.e., $h(X_1+\ldots+X_n)<\infty$).
Define the set function $\nu_{\bf f}:2^{[n]}\setminus\nullset \ra\Rpl$ by
\be\label{nu}
\nu_{\bf f}(\setS)=\calN\bigg(\sum_{i\in\setS} X_i\bigg) ;
\ee 
we will also denote $\nu_{\bf f}$ by $\nu_{(X_1, \ldots, X_n)}$ when convenient.
The $d$-dimensional Stam region is the subset of $\RL^{2^{n}-1}$ given by
\ben
\Gamma_d(n)=\{ \nu_{\bf f} : {\bf f}\in {\bf F}_{d,n} \}.
\een
The Stam region is defined by 
\ben
\Gamma(n)= \cup_{d\in\Nat} \Gamma_d(n).
\een  
\end{defn}

We name these regions after A. J. Stam in honor of his pioneering role 
\cite{Sta59} in the study of entropy power and its applications.
One can extend the domain of $\nu_{\bf f}$ to the full Boolean lattice $2^{[n]}$
in a natural way by setting $\nu_{\bf f}(\phi)=0$ for every ${\bf f}\in {\bf F}_{d,n}$;
this would make the Stam region a subset of $\RL^{2^{n}}$ instead of $\RL^{2^{n}-1}$,
but we avoid such an extension since it is trivial.

Any inequality that relates entropy powers of different subset sums (usually
called an ``entropy power inequality'' or EPI) gives a bound on the Stam region.
Conversely, knowing the Stam region is equivalent, in principle, to knowing all EPI's that hold
and all that do not.

The objective of this note is to develop a better understanding of the Stam region.
Our first result is the best outer bound known thus far for it. 
In particular, we show that for any dimension $d$, the entropy power is
``fractionally superadditive'', settling a conjecture made implicitly in \cite{MB07} and 
explicitly in \cite{Mad08:game}. We first explain what the term means. 

Let $\collS$ be a hypergraph on $[n]$,
i.e., let $\collS$ be a collection of nonempty subsets of $[n]$.
Given $\collS$, a function $\beta:\collS \to \Rpl$ 
is a {\em fractional partition}, if  for each $i\in [n]$, we have
$\sum_{\setS\in \collS:i\in \setS} \bs = 1$. If there exists
a fractional partition $\beta$ for $\collS$ that is $\{0,1\}$-valued,
then $\beta$ is the indicator function for a partition of the set $[n]$
using a subset of $\collS$; hence the terminology.
Note that we do not need to make reference to a hypergraph $\collS$ 
in order to define a fractional partition; the domain can always be extended
to all nonempty subsets of $[n]$ by setting $\bs=0$ for $s\notin \collS$.

We say that a set function $v:2^{[n]}\ra\Rpl$ is {\it fractionally superadditive} if 
\be\label{fsa}
v([n]) \geq \sumS \bs v(\setS)
\ee
holds for every fractional partition $\beta$ using any hypergraph $\collS$.
Write $\Gamma_{FSA}(n)$ for the class of all  fractionally superadditive 
set functions  $v$ with $v(\phi)=0$.

\begin{thm}\label{thm:fsa}
For every $n\in\Nat$, $\Gamma(n) \subset \Gamma_{FSA}(n)$.
\end{thm}

The set function $v:2^{[n]}\ra\Rpl$ is said to be {\it supermodular} if 
\ben
v(\setS\cup\setT) + v(\setS\cap\setT)  \geq v(\setS) + v(\setT)  
\een
for all sets $\setS,\setT\subset [n]$. Write $\Gamma_{SM}(n)$ for the class of all supermodular
set functions $v$ with $v(\phi)=0$.

It is known \cite{MP82, MT10} that $\Gamma_{SM}(n)\subsetneq \Gamma_{FSA}(n)$.
In \cite{Mad08:game}, it was asked whether in fact
the entropy power is supermodular, i.e., whether the set function $\nu_{\bf f}\in \Gamma_{SM}(n)$ 
for every ${\bf f}\in {\bf F}_{d,n}$. Our second result is to show
that the answer to this question is no in general.

\begin{thm}\label{thm:sm}
For any $n\geq 3$, $\Gamma(n) \nsubseteq \Gamma_{SM}(n)$.
\end{thm}

The Stam region has some pleasing geometric properties. Unfortunately we have
not yet been able to determine if the closure of the Stam region is convex; however it does have
some restricted convexity properties. We say that a set $A\subset \RL_+^k$ is 
{\it logarithmically convex} if for any two points $x=(x_1,\ldots,x_k)$ and
$y=(y_1,\ldots,y_k)$ in  $A$, the point $(x_1^\lam y_1^{1-\lam},\ldots,x_k^{\lam} y_k^{1-\lam})\in A$
for each $\lam \in [0,1]$. 
We say that a set $A\subset \RL^k$ is {\it orthogonally convex} if 
any line segment parallel to any of the coordinate axes connecting two points of $A$ 
lies totally within $A$. It is trivial to check that any subset of $\RL_+^k$
that is logarithmically convex is necessarily orthogonally convex.

\begin{thm}\label{thm:cvx}
The closure $\overline{\Gamma(n)}$ of the Stam region $\Gamma(n)$ 
is a logarithmically convex cone in $\RL^{2^{n}-1}$. 
Consequently, $\overline{\Gamma(n)}$ is orthogonally convex.
\end{thm}

For the case where one has only two random variables, we can give a complete description
of the Stam region.

\begin{thm}\label{thm:dim2}
$\overline{\Gamma(2)}=
\Gamma_{FSA}(2)$. In particular, 
$\overline{\Gamma(2)}$ is a closed, convex, polyhedral cone in $\RL_{+}^3$.
\end{thm}

The general problem is significantly more complicated.

\begin{thm}\label{thm:dim3}
For $n\geq 3$,
$\overline{\Gamma(3)}\subsetneq \Gamma_{FSA}(3)$.
\end{thm}

The closure of the Stam region, which features in the preceding three theorems,
is in fact closely related to the ``Stam region'' involving entropy rates of a 
class of stochastic processes. By a $\RL$-valued stochastic process,
we will mean a discrete-time stochastic process $X=(X_1, X_2, \ldots)$, 
with each $X_i$ taking values in $\RL$. (More precisely, we may
interpret this process as a measurable function from the underlying
probability space to the set $\RL^{\Nat}$ equipped with the cylindrical
$\sigma$-algebra.)
Recall that the {\it entropy rate} of a $\RL$-valued stochastic process
is defined by
\ben
\bar{h}(X)= \lim_{d\ra \infty} \frac{h(X_1, \ldots, X_d)}{d} ,
\een
if the limit exists. 

\begin{defn}
Let ${\bf F}_{n}$ be the collection of all $n$-tuples  ${\bf X}=(X^{(1)}, \ldots, X^{(n)})$
of independent $\RL$-valued stochastic processes, such that each $X^{(i)}$ as well as
their sum has finite entropy rate
(i.e., $\bar{h}(X^{(1)}+ \ldots+ X^{(n)})<\infty$).
Define the set function $\nu_{\bf X}:2^{[n]}\setminus\nullset \ra\Rpl$ by
\be\label{nu}
\nu_{\bf X}(\setS)=\exp{ \bigg\{ 2\bar{h}\bigg(\sum_{i\in\setS} X^{(i)}\bigg) \bigg\} }.
\ee 
The $\infty$-dimensional Stam region is the subset of $\RL^{2^{n}-1}$ given by
\ben
\Gamma_{\infty}(n)=\{ \nu_{\bf X} : {\bf X}\in {\bf F}_{n} \}.
\een
\end{defn}

Now we can relate the finite-dimensional  and infinite-dimensional Stam regions.

\begin{thm}\label{thm:infStam}
For any $n\in \Nat$,
$\overline{\Gamma(n)}= \Gamma_{\infty}(n)$.
\end{thm}

This gives a pleasing ``physical'' interpretation of the closure of the Stam region--
the closure can be thought of not just as a topological operation on a set that has
intrinsic probabilistic motivation, but as an extension of the definition of the Stam
region to random vectors of  dimension $\infty$.

Let us note that the restriction 
$h(X_1+\ldots+X_n)<\infty$ in the definition of the Stam region 
is not essential. One can define the {\it extended $d$-dimensional Stam region}
as 
\ben
\tilde{\Gamma}_d(n) = \{ \nu_{\bf f} : {\bf f}\in \tilde{\bf F}_{d,n} \}, 
\een
where $\tilde{\bf F}_{d,n}$ is the collection of all $n$-tuples  ${\bf f}=(f_1, \ldots, f_n)$
of probability density functions $f_i$ on $\RL^d$ such that $h(f_i)$ exists for each $i\in [n]$.
The  {\it extended Stam region} would then be defined by
\ben
\tilde{\Gamma}(n) = \cup_{d\in\Nat} \tilde{\Gamma}_d(n) ;
\een
both $\tilde{\Gamma}_d(n)$ and $\tilde{\Gamma}(n)$ are then 
subsets of $\bar{\RL}_{+}^{2^{n}}$, with the extended nonnegative real numbers
$\bar{\RL}_{+}=[0,\infty]$ replacing $\Rpl$. All our theorems, with straightforward
minor modifications, can be stated for the extended Stam regions.

After developing some combinatorial preliminaries in Section~\ref{sec:comb-prelim},
we prove the above theorems in Section~\ref{sec:good-epi}. 
Although Theorem~\ref{thm:sm} asserts that supermodularity is false for the Stam region,
the proof given in Section~\ref{sec:good-epi} does not shed light on the whether
supermodularity might in fact hold in the one-dimensional case (real-valued random variables); 
this question is discussed in Section~\ref{sec:conj}. 
Section~\ref{sec:rmks} contains some additional remarks of possible interest.

\section{Combinatorial Preliminaries}
\label{sec:comb-prelim}

The preliminaries we will need range from 
a lemma about fractional superadditivity that we could not find
explicitly stated in the literature, to some facts about
supermodular functions on lattices.

We now state and prove a useful lemma about fractional superadditivity
that is at the heart of our proof of Theorem~\ref{thm:fsa}. We will use in the proof
the following standard fact about a vertex or extreme point of a polytope: a point in a polytope is an extreme point if and only if
it is the unique meeting point of several faces (i.e., a set of faces intersects in a singleton, containing that point). 
Indeed, if it were not unique, then it could not be an extreme point, since there would
be a line segment in the polytope containing it.

\begin{prop}\label{prop:suff}{\sc [A sufficient condition for Fractional Superadditivity]}
Consider a set function $v:2^{[n]}\ra\Rpl$.
Let $\collS$ be a $r$-regular multihypergraph on $[n]$,
i.e., let $\collS$ be a collection of subsets of $[n]$ (possibly repeated),
such that every index $i$ lies in exactly $r$ of the elements of $\collS$.
Suppose $v$ satisfies, for any $r\in \Nat$,
\ben
v([n]) \geq \rth \sumS v(\setS)
\een
where $\collS$ is any $r$-regular multihypergraph on $[n]$.
Then, 
\be\label{fsa}
v([n]) \geq \sumS \bs v(\setS)
\ee
holds for every fractional partition $\beta$ using any multihypergraph $\collS$
on $[n]$.
\end{prop}

\begin{proof}
Denote by $\one_{\setS}$ the indicator function of the subset $\setS$ defined on the domain $[n]$,
i.e., $\one_{\setS}(i)= 1$ if $i\in \setS$, and $\one_{\setS}(i)= 0$ otherwise.
Consider the space of all fractional partitions on $[n]$, i.e.,
\ben
\calB=\bigg\{\beta:2^{[n]}\setminus\nullset\ra\Rpl \bigg| \sum_{\setS\subset [n]\setminus\nullset}\bs \one_{\setS}=\one_{[n]} \bigg\}.
\een
Clearly, $\calB$ can be viewed as a subset of the Euclidean space of dimension $2^{n}-1$
(since each point of it is defined by  $2^{n}-1$ real numbers). Furthermore,
$\calB=\calB'\cap \calO_{+}$, where
\ben
\calB'=\bigg\{\beta:2^{[n]}\setminus\nullset\ra\RL \bigg| \sum_{\setS\subset [n]\setminus\nullset}\bs \one_{\setS}=\one_{[n]} \bigg\}
\een
is an affine subspace of dimension $2^{n}-1-n$ and 
$\calO_{+}=\{\beta|\bs\geq 0 \,\,\forall \setS\in 2^{[n]}\setminus\nullset\}$ 
is the closed positive orthant. 
In particular, $\calB$ is a non-empty, compact, convex set (in fact, a closed polytope),
so that by the Krein-Milman theorem, $\calB$ is the convex hull of its extreme points.
Thus to prove \eqref{fsa} for every $\beta\in\calB$, it is sufficient to prove
\eqref{fsa} for every $\beta\in\Ex(\calB)$, where $\Ex(\calB)$ denotes the
set of extreme points of $\calB$.

For a fractional partition $\beta$, its {\it support} is defined as the collection of subsets $\setS$ of $[n]$
such that $\bs>0$. Given a hypergraph $\collS$, the set of fractional partitions supported by $\collS$ is clearly the set 
of positive solutions of the linear equation
\be\label{lineq}
M_{\collS}\beta=\one ,
\ee
where $M_{\collS}$ is the $n\times |\collS|$ 0-1 matrix defined by
$M_{i,\setS}=\one_{i\in\setS}$ for $i\in[n], \setS\in\collS$,
and $\one$ is the column vector in $\RL^{n}$ consisting of all ones. 
In general, the number of  fractional partitions supported by $\collS$ could be either 0, 1 or infinite,
depending on $\collS$.

Every face of the polytope $\calB$ corresponds to one of the inequality constraints
being tight (i.e., $\bs=0$ for at least one set $\setS$). 
Given an extreme point $\beta$ of  the polytope $\calB$, we know it is the  
{\it unique} meeting point of several faces; let these faces correspond to setting $\bs=0$ for $\setS$ lying
in the collection of sets $\collS'$. Then the complement $\collS$ of $\collS'$
is the support of $\beta$, and $\beta$ must be the unique fractional partition
supported by $\collS$. By the previous paragraph, we know this means that $\beta$ is the 
unique, strictly positive solution of the equation \eqref{lineq}. 
Consequently one must have $|\collS|\leq n$, with precisely
$|\collS|$ of the $n$ rows of $M_{\collS}$ being linearly independent,
so that $M_{\collS}$ has full rank. Since $M_{\collS}$ has only integer entries,
a little bit of thought shows that performing Gaussian elimination to solve \eqref{lineq}
will result in its unique solution having only rational entries. Thus we have deduced
that any fractional partition in $\Ex(\calB)$ has rational\footnote{Another way of seeing this would be to observe that 
$\beta=M_{\collS}^{+}\one$, where $M_{\collS}^{+}$ is the Moore-Penrose pseudo-inverse
of $M$. Then the rationality of entries of $M_{\collS}$ implies that of the entries of $M_{\collS}^{+}$,
and hence the rationality of $\beta$.} coefficients.

By writing all the coefficients of $\beta$ with a common denominator,
one sees that \eqref{fsa} may be written as
\ben
v_n([n]) \geq \frac{1}{R}\sumS c_{\setS} v_n(\setS) ,
\een
where $c_{\setS}$ is a positive integer. One may write this as
\ben
v_n([n]) \geq \frac{1}{R}\sum_{\setS\in\collS''}  v_n(\setS) ,
\een
where $\collS''$ is the multihypergraph with 
$c_{\setS}$ copies of the set $\setS$.
Note that $\collS''$ is clearly $R$-regular. 
\end{proof}
\vspace{.1in}

\begin{rmk}
The argument we use to characterize $\Ex(\calB)$ is similar
to that in Gill and Gr\"unwald \cite{GG08} and seems also implicit in Friedgut and Kahn \cite{FK98}; 
it likely goes back to the early literature on cooperative game theory (see, e.g., Shapley \cite{Sha67}), 
although we were unable to find the precise statement we wanted in these early references.
Incidentally a fractional partition is interpreted in \cite{GLR97, GG08} as a ``coarsening at random'' or CAR mechanism;
this is a probabilistic rule that replaces any point $x$ in the ground set $[n]$ with a subset $A$ of $[n]$ containing $x$,
in such a way that the probability of observing $A$ is the same for all $x$ that are contained in $A$.
\end{rmk}

The last observation we need concerns supermodular functions; 
first we need a definition of supermodularity for 
functions on $\RL^n$ and not just set functions.

\begin{defn}
A function $f:\RL_{+}^{n} \ra\RL$ is supermodular if
\ben
f(x) + f(y) \leq f(x\vee y) + f(x \wedge y)
\een
for any $x,y \in \mathbb{R}_+^n$, where $x \vee y$ denotes the componentwise maximum of $x$ and 
$y$ and $x \wedge y$ denotes the componentwise 
minimum of $x$ and $y$.
\end{defn}

The fact that supermodular functions are closely related to functions 
with increasing differences is classical (see, e.g., \cite{Top98:book}, which describes
more general results involving arbitrary lattices), but we give the proof for
completeness. 

\begin{prop}\label{prop:diff}
Suppose a function $f:\RL_+^n\ra\RL$ is in $C^2$, i.e., it is twice-differentiable with
a continuous Hessian matrix. Then $f$ is supermodular if and only if
\ben
\frac{\partial ^2 f(x)}{\partial x_j \partial x_i}\geq 0
\een
for every distinct $i, j \in [n]$, and for any $x\in \RL_+^n$.
\end{prop}

\begin{proof}
Suppose $f:\RL^n\ra\RL$ is supermodular. Let $i,j\in [n]$
with $i<j$, and pick a set of coordinates $a_k$ for each $k\in A$,
where $A=[n]\setminus \{i,j\}$.
For any $x\in\RL^n$ with $x_k=a_k$ for  each $k\in A$,
set 
\ben
g(x_i, x_j)=f(x) .
\een
Then for $y_i < z_i$ and $y_j <z_j$, 
\ben\begin{split}
g(y_i, z_j) - g(y_i, y_j) 
&= g(y_i, z_j) - g(y_i \wedge z_i , y_j \wedge z_j) \\
&\leq  g(y_i \vee z_i , y_j \vee z_j)- g(z_i, y_j) \\
&= g(z_i, z_j)- g(z_i, y_j) ;
\end{split}\een
in other words, differences of $g$ (and hence of $f$) in the $j$-th coordinate
are increasing in the $i$-th coordinate. If $f$ is $C^2$, this implies
that 
\ben
\frac{\partial ^2 f(x)}{\partial x_j \partial x_i}\geq 0 
\een
on the interior of the domain. Since the same argument can be repeated
for each pair $i, j$ of distinct indices, we obtain non-negativity of all off-diagonal
entries of the Hessian matrix of $f$.

Now suppose
\ben
\frac{\partial ^2 f(x)}{\partial x_j \partial x_i}\geq 0
\een
for any $x\neq 0$. By standard calculus, this implies
that $f$ has increasing differences, in the sense that for any pair $i\neq j$
of distinct indices, and for every $a=(a_k: k\in A)$
where $A=[n]\setminus \{i,j\}$, one has that differences
in the $j$-th coordinate, namely
\ben
f(a,x_i,x_j) - f(a,x_i, \tilde{x_j})
\een
with $x_j > \tilde{x_j}$, are increasing as functions of $x_i$.
(We have abused notation for convenience here since $f$ may not be symmetric in its arguments,
but we are implicitly assuming that $x_i$ is the $i$-th argument and ${x_j}$
is the $j$-th argument when we write $f(a,x_i,x_j)$, etc.)
Thus we have
\ben\begin{split}
f(y)- f(y\wedge z) 
&= \sum_{i\in [n]} \bigg[ f(y_1^{i-1}, y_i, y_{i+1}^n \wedge z_{i+1}^n) - f(y_1^{i-1}, y_i \wedge z_i,  y_{i+1}^n \wedge z_{i+1}^n) \bigg] \\
&\leq  \sum_{i\in [n]} \bigg[ f(y_1^{i-1}\vee z_1^{i-1}, y_i, z_{i+1}^n) - f(y_1^{i-1}\vee z_1^{i-1}, y_i \wedge z_i,  z_{i+1}^n) \bigg] ,
\end{split}\een
where the inequality follows from the increasing differences property of $f$.  Now we can apply the
trivial fact that for any function $g$ on the real line, one has $g(x)+g(y)=g(x\vee y) +g(x\wedge y)$
to deduce that the last expression equals
\ben\begin{split}
\sum_{i\in [n]} \bigg[ f(y_1^{i-1}\vee z_1^{i-1}, y_i\vee z_i, z_{i+1}^n) - f(y_1^{i-1}\vee z_1^{i-1}, z_i,  z_{i+1}^n) \bigg] ,
\end{split}\een
which telescopes down to $f(y\vee z) -f(z)$. Combining the preceding two displays, we have proved that 
\ben
f(y)- f(y\wedge z) 
\leq f(y\vee z) -f(z) ,
\een
which is precisely supermodularity of $f$.
\end{proof}

Our final very simple lemma connects supermodularity for functions defined on $\RL_{+}^n$
to supermodularity for set functions. For a set $s\subset [n]$, we use $e(s)$ to denote
the vector in $\RL^n$ such that for each $i\in [n]$, the $i$-th coordinate of $e(s)$ is 1
if and only if $i\in s$ (or in other words, $e_i(s):= \one_s (i)$).

\begin{lem}\label{lem:set-rn}
If $f:\RL_{+}^n\ra\RL$ is supermodular, and we set $\bar{f}(s):=f(e(s))$ for each $s\subset [n]$,
then $\bar{f}$ is a supermodular set function.
\end{lem}

\begin{proof}
Observe that
\ben\begin{split}
\bar{f}(s\cup t) + \bar{f}(s\cap t) 
&= f(e(s\cup t)) + f(e(s\cap t)) \\
&= f(e(s) \vee e(t)) + f(e(s)\wedge e(t)) \\
&\geq f(e(s))+f(e(t)) \\
&= \bar{f}(s) + \bar{f}(t) .
\end{split}\een
\end{proof}

\section{Proofs} 
\label{sec:good-epi}

\subsection{Proof of Theorem~\ref{thm:fsa}}

In \cite{MB07}, the following EPI for an arbitrary multihypergraph $\collS$ on $[n]$
was demonstrated. If $r$ is the maximum number of subsets in $\collS$ 
in which any one index $i$ can appear, for $i=1,\ldots,n$,
then
\be\label{r-epi}
\calN(X_{1}+\ldots+X_{n}) \geq \frac{1}{r} \sum_{\setS\in \collS} 
\calN \bigg( \sum_{j\in\setS} X_{j}\bigg) .
\ee
In other words, for any ${\bf f}\in {\bf F}_{d,n}$, $\nu_{\bf f}$ satisfies
the assumption in Proposition~\ref{prop:suff}, and hence we have that
$\nu_{\bf f}$ is fractionally superadditive.

\subsection{Proof of Theorem~\ref{thm:sm}} 
\label{sec:bad-epi}

Observe that if the function 
$\nu_{\bf f}(\setS)=\calN(\sum_{i\in\setS} X_i)$
were supermodular, then specializing to multivariate Gaussians
would imply a similar supermodularity for the $d$-th root
of the determinant of sums of positive definite matrices.
To be precise, consider the set function 
\ben
\nu_{G} (\setS) = \text{det}\bigg(\sum_{k\in \setS} M_k\bigg)^\frac{1}{d} ,  \quad\setS\subset [n],
\een
where  $M_1,M_2,\cdots,M_n$ are $d \times d$ positive definite matrices. 
(Here we use $\nu_G$ to indicate that this is the function $\nu_{\bf f}$ specialized to 
Gaussians, i.e., when $f_i$ is the density of the $N(0,M_i)$ distribution.)
We will show that $\nu_G$ is not supermodular for $d=2$.
\par\vspace{.05in}
\noindent{\bf First proof.} We exhibit a numerical counterexample. 
Consider
\ben
A=\left[\begin{array}{cc}2 &0 \\0 &\half\end{array}\right], \, 
B=\left[\begin{array}{cc}\half &0 \\0 &2\end{array}\right], \, 
C=\left[\begin{array}{cc}\epsilon &0 \\0 &\epsilon\end{array}\right].
\een
Then it is easily seen that
\ben\begin{split}
&\det(A)^{1/2}=1, \quad \det(A+B+C)^{1/2}=2.5+\epsilon \\
&\det(A+B)^{1/2}=2.5, \quad\det(A+C)^{1/2}=\sqrt{1+2.5\epsilon+\epsilon^2}
\end{split}\een
Observe that when $\epsilon$ is small, $\det(A+C)^{1/2}=1+1.25\epsilon+o(\epsilon)$, so that we can
arrange for  $\det(A+B)^{1/2}+\det(A+C)^{1/2}=3.5+1.25\epsilon+o(\epsilon)$ to exceed $\det(A)^{1/2}+ \det(A+B+C)^{1/2}=3.5+\epsilon$.
\par\vspace{.05in}
\noindent{\bf Second proof.} We present a second, more complicated, proof in an attempt to give some additional
insight into why supermodularity fails.
To prove that $\nu_G$ is not supermodular, we first show that its
continuous analogue is not supermodular (in the continuous sense). In other words, let 
$v:\Rpl^n \rightarrow \Rpl$ be defined by
\ben
v(x) = \text{det}^\frac{1}{d}\bigg(\sum_{k\in [n]} x_kM_k\bigg) ,
\een
where  $M_1,M_2,\cdots,M_n$ are $d \times d$ positive definite matrices. 
We will show that the function $v$ is not supermodular, i.e., there are 
$x,x'\in \Rpl^n$ 
such that the following inequality is violated
\be\label{Eq:Convexity}
v(x)+v(x') \leq v(x \vee x') + v(x \wedge x')
\ee
where $x \vee x'$ denotes the componentwise maximum and $x \wedge x'$ 
denotes the componentwise minimum of $x$ and $x'$.

To show that \eqref{Eq:Convexity} is violated, it suffices to show, 
thanks to Proposition~\ref{prop:diff}, that $v$ is $C^2$
(which is immediate from the fact that the determinant of a linear combination
of matrices is a polynomial in the coefficients) and that 
\be\label{cont-crit}
\frac{\partial ^2 v(x)}{\partial x_j \partial x_i}< 0
\ee 
for some $x \in \Rpl^{n}$. 
Setting $M=\sum_{k\in [n]} x_kM_k$, we note that
\be\label{Eq:PartialDerivative}\begin{split}
\frac{\partial^2 v(x)}{\partial x_j \partial x_i} &= 
\frac{1}{{d}}(\text{det}\, M)^\frac{1}{{d}} \times \bigg[\frac{1}{{d}}\tr\{M^{-1}M_i\} \tr\{M^{-1}M_j\} - \tr(M^{-1} M_j M^{-1} M_i) \bigg].
\end{split}\ee
However, it is easy to show (see, e.g., Zhang \cite{Zha99:book}, page 166)
that there are ${d} \times {d}$ positive definite matrices 
$A$ and $B$ for which 
\ben
\frac{1}{{d}}\tr(A)\tr(B) < \tr(AB) .
\een
Hence, the last term in~\eqref{Eq:PartialDerivative} can be negative. As a numerical example, 
consider ${d}=2$ and
\ben
A=\left[\begin{array}{cc}3 &1 \\1 &1\end{array}\right], \, 
B=\left[\begin{array}{cc}2 &3 \\3 &7\end{array}\right].
\een 
It then holds that $\frac{1}{2}\tr(A)\tr(B)=18$ whereas $\tr(AB)=19$.

Finally, note that the violation of supermodularity in the discrete domain follows from this result
of  violation of supermodularity in continuous domain. Indeed,
we have shown that there is $x\in \Rpl^{n}$ such that 
\eqref{cont-crit} is true. It then follows that there are 
$s_i, s_j>0$ such that
\ben\label{Eq:PartialDerivate}\begin{split}
&v(x_1, \ldots, x_i+s_i, \ldots, x_j, \ldots, x_n) + v(x_1, \ldots, x_i, \ldots, x_j+s_j, \ldots, x_n) \\
&> v(x_1, \ldots, x_i+s_i, \ldots, x_j+s_j, \ldots, x_n)+ v(x).
\end{split}\een
Consider $A=s_i M_i$, 
$B=s_j M_j$, 
$C=\sum_{k\in [n]} x_kM_k$. 
Evidently, $A, B, C$ are positive definite. However, according to the above inequality 
and the definition of the function $v$, we have shown that
\ben\begin{split}
&\text{det}\left(A+C\right)^\frac{1}{{d}} + \text{det}\left(B+C\right)^\frac{1}{{d}} \\
&> \text{det}\left(A+B+C\right)^\frac{1}{{d}} + \text{det}\left(C\right)^\frac{1}{{d}},
\end{split}\een
so that neither $\nu_G$ nor $\nu$ is supermodular.

\subsection{Proof of Theorem~\ref{thm:cvx}}

The fact that each $\Gamma_d(n)$ is a cone
follows just from scaling. By definition, any point $z\in \Gamma_d(n)$
is just $\nu_{(X_1,\ldots, X_n)}$ for some independent random vectors
$X_1, \ldots, X_n$ (each of which is in $\RL^d$). For any $\lambda>0$, the point $\lambda z$ 
is then simply $\nu_{(\sqrt{\lambda} X_1, \ldots, \sqrt{\lambda} X_n)}$
and hence also lies in $\Gamma_d(n)$.
Since each $\Gamma_d(n)$ is a cone, so are $\Gamma(n)$
and $\overline{\Gamma(n)}$. 

What remains is to show that $\overline{\Gamma(n)}$ is a logarithmically convex set.
To do this, we will need the following lemma.

\begin{lem}\label{lem:dense}
For any given $d, d'\in\Nat$, the set
\ben
Q_{d,d'}=\bigg\{ \frac{md}{md+ld'}: m, l\in \Zpl \bigg\}
\een
is a dense subset of $[0,1]$.
\end{lem}

\begin{proof}
Letting $z(m,l)=\frac{md}{md+ld'}$, the image $Q_{d,d'}$ of $\Zpl\times\Zpl$ under $z$ is 
clearly contained in the interval $[0,1]$. Since $[0,1]$ is closed, we also have
$\overline{Q_{d,d'}}\subset [0,1]$; thus what remains to be proved is that every point of $[0,1]$
is in the closure of $Q_{d,d'}$. 

Let $x$ be an arbitrary rational number in $[0,1]$; then $x$ is of the form $\frac{p}{p+q}$,
where $p, q\in \Zpl$. Writing
\ben
z(m,l)=\frac{d}{d+\frac{l}{m}d'} = \frac{1}{1+\frac{l}{m}\cdot\frac{d'}{d}} ,
\een
we see that $x$ is in the image of $Q_{d,d'}$ of $\Zpl\times\Zpl$ under $z$
by taking $l=qd$ and $m=pd'$. In other words, we have $\mathbb{Q}\cap [0,1]\subset Q_{d,d'}$,
which-- by the density of the rationals in the reals-- implies that $[0,1]\subset \overline{Q_{d,d'}}$.

Thus it is demonstrated that $\overline{Q_{d,d'}}=[0,1]$.
\end{proof}

\begin{rmk}
Lemma~\ref{lem:dense} is related to interesting problems in number theory, such as the
coin problem (sometimes called the diophantine Frobenius problem), which goes back to 1882 and asks for the largest 
integer that cannot be written as a nonnegative linear combination of a given set of natural numbers.
(A theorem of Schur ensures that there is such a largest integer.) 
The book of Ram\'irez Alfons\'in \cite{Ram05:book} contains many more details.
\end{rmk}

To show that $\overline{\Gamma(n)}$ is a logarithmically convex set,
let us start with two points in $\Gamma(n)$, say $\nu_{{\bf X}}:=\nu_{(X_1, \ldots, X_n)}\in \Gamma_d(n)$
and $\nu_{{\bf Y}}:=\nu_{(Y_1, \ldots, Y_n)}\in \Gamma_{d'}(n)$, where our
notation is as explained in Definition~\ref{defn:stam}. Let ${\bf Z}^{(m,l)}$ be the vector
formed by stacking $m$ independent copies of  ${\bf X}$ with $l$ independent
copies of ${\bf Y}$; in other words,
\ben
{\bf Z}^{(m,l)}= 
\left( \begin{array}{c}
{\bf X}^{(1)} \\
. \\
. \\
{\bf X}^{(m)} \\
{\bf Y}^{(1)} \\
. \\
. \\
{\bf Y}^{(l)} \end{array} \right)
= \left( \begin{array}{ccc}
X_1^{(1)} & \ldots & X_n^{(1)} \\
.  & \ldots & .\\
.  & \ldots & .\\
X_1^{(m)}  & \ldots & X_n^{(m)} \\
Y_1^{(1)}  & \ldots & Y_n^{(1)} \\
. \\
. \\
Y_1^{(l)}  & \ldots & Y_n^{(l)} \end{array} \right)
,
\een
where each $X_i^{(k)}\in \R^d$ and each $Y_i^{(k)}\in \R^{d'}$.
Thus, denoting the $i$-th column of ${\bf Z}^{(m,l)}$ by $Z_i^{(m,l)}$,
we see that each $Z_i^{(m,l)}$ is a random vector in $\R^{md+ld'}$.
For each $\setS\subset [n]$, we have
\ben\begin{split}
\nu_{(Z_1^{(m,l)}, \ldots, Z_n^{(m,l)})} (\setS)
&= \calN\bigg( \sum_{i\in\setS} Z_i^{(m,l)}\bigg) \\
&= \calN\bigg( \sum_{i\in\setS} X_i^{(1)}, \ldots, \sum_{i\in\setS} X_i^{(m)}, 
\sum_{i\in\setS} Y_i^{(1)}, \ldots, \sum_{i\in\setS} Y_i^{(l)} \bigg) \\
&= \calN\bigg( \sum_{i\in\setS} X_i \bigg)^{\frac{md}{md+ld'}}
\calN\bigg( \sum_{i\in\setS} Y_i \bigg)^{\frac{ld'}{md+ld'}} \\
&=\nu_{{\bf X}}(\setS)^\lam \nu_{{\bf Y}}(\setS)^{1-\lam} ,
\end{split}\een
where the third equality follows from the assumed independence
and the fact that the dimension appears in the exponent in the definition of entropy power.

Thus we have shown that if $a\in \Gamma_d(n)$ and $b\in \Gamma_{d'}(n)$, then
\ben
\{ a^\lam b^{1-\lam}: \lam\in Q_{d,d'}\}
\een
is a subset of $\Gamma(n)$. Invoking the density of $Q_{d,d'}$ in $[0,1]$ (Lemma~\ref{lem:dense})
and the continuity of the map $\lam\mapsto a^\lam b^{1-\lam}$,
we see that 
\ben
\{ a^\lam b^{1-\lam}: \lam\in [0,1]\}
\een
is a subset of $\overline{\Gamma(n)}$, which is precisely the claimed logarithmic convexity.

\subsection{Proof of Theorem~\ref{thm:dim2}}

The main challenge here is to show that $\Gamma_{FSA}(2) \subset \overline{\Gamma(2)}$.
Since $\Gamma_{FSA}(2) = \{ (u_{\{1\}}, u_{\{2\}}, u_{\{1,2\}})\in \RL_+^3:  u_{\{1,2\}} \geq u_{\{1\}}+ u_{\{2\}} \}$,
it suffices to show that the ray 
$R_{u_{\{1\}}, u_{\{2\}}}$ defined as $\{(u_{\{1\}}, u_{\{2\}}, u_{\{1,2\}}): u_{\{1,2\}}\geq u_{\{1\}}+ u_{\{2\}} \}$
lies in $\overline{\Gamma(2)}$ for any given $u_{\{1\}}, u_{\{2\}} \in \RL_+$.
In other words, we wish to show that for any given $u_{\{1\}}, u_{\{2\}} \in \RL_+$,
and for a dense subset of $u_{\{1,2\}}$ such that $u_{\{1,2\}}\geq u_{\{1\}}+ u_{\{2\}}$, 
there exist independent random vectors $X, Y$ such that
$\calN(X)=u_{\{1\}}$, $\calN(Y)=u_{\{2\}}$, and $\calN(X+Y)=u_{\{1,2\}}$.

We find it useful to adapt a construction of Bobkov and Chistyakov \cite[Example 3]{BC15:1}
for our purposes. Consider the uniform density $\tilde{f}$ on a set  
\ben
A=\cup_{n\in\Nat} (2^n, 2^n + a_n) ,
\een
where $a_n\geq 0$ for each $n$, and $\sum_{n\in\Nat} a_n =1$. Alternatively,
one can write $\tilde{f}$ in the form
\ben
\tilde{f}(x) = \sum_{n\in\Nat} a_n \tilde{f}_n(x) ,
\een
where $\tilde{f}_n(x) = a_n^{-1} \one_{(2^n, 2^n + a_n)} (x)$. 
If $X, X'$ are i.i.d. and have density $\tilde{f}$, then $h(X)=0$ since $A$ has length 1,
and \cite{BC15:1} showed that 
\ben
h(X+X')&= -2\log 2 + 2\log 2 \sum_{n\in\Nat} s_n a_n + \sum_{n\in\Nat} \big(n-\frac{1}{2}\big) a_n^2 \\
&+\sum_{n\in\Nat} a_n^2 \log \frac{1}{a_n} + 2 \sum_{n=2}^{\infty} s_{n-1} a_n \log \frac{1}{a_n} ,
\een
where $s_n=a_1+\ldots+a_n$.
By choosing the weights $a_n$ appropriately, one can clearly make $h(X+X')$ 
larger than any constant one might pick.\footnote{Indeed, one can force
$h(X+X')$ to be infinite; note that $h(X+X')<\infty$ if and only if
\ben
\sum_{n\in\Nat} a_n \log \frac{1}{a_n} <\infty .
\een 
This was what motivated \cite{BC15:1} to consider this construction, since they wished to exhibit 
a density $\tilde{f}$ with finite differential entropy whose self-convolution has infinite
differential entropy.}
For example, for $N$ chosen large enough, one may take $a_n=1/N$ for $1\leq n\leq N$, and $a_n=0$ for
$n>N$.
By scaling (i.e., considering the random variables $Y=bX$ and $Y'=bX'$
for $b>0$), one has examples where $h(Y)=\log b$ is any specified real number, and 
$h(Y+Y')$ is larger than an arbitrary specified constant. 

Now suppose $u_{\{1\}}, u_{\{2\}}$ are given, and assume without loss of generality that $u_{\{1\}}\leq u_{\{2\}}$.
Fix an arbitrarily large positive constant $C$.
By the preceding paragraph, there exists a random vector $X$ such that $h(X)=s:= \frac{d}{2}\log u_{\{1\}}$,
and $\calN(X+X')>C$, where $X'$ is an independent copy of $X$.
Observe that the function $u\mapsto h(X'+\sqrt{u}Z)$,
where $Z$ is an independent standard Gaussian, maps $[0,\infty)$ to $[s,\infty)$,
and is also smooth and monotonically increasing (as quantified by the classical de Bruijn identity).
Consequently, there must exist a $u>0$ such that $h(X'+\sqrt{u}Z)=t:=\frac{d}{2}\log u_{\{2\}}$. 
Setting $Y=X'+\sqrt{u}Z$, observe that $\calN(X)= u_{\{1\}}$, $\calN(Y)= u_{\{2\}}$,
and $\calN(X+Y)=\calN(X+X'+\sqrt{u}Z) \geq \calN(X+X')>C$.

From the preceding paragraph, for any given $u_{\{1\}}, u_{\{2\}}>0$,
and for arbitrary $C$, there exists $C'>C$ such that 
$(u_{\{1\}}, u_{\{2\}}, C')\in \Gamma_1(2) \subset \overline{\Gamma(2)}$.
On the other hand, by considering one-dimensional Gaussians of appropriate variances,
clearly $(u_{\{1\}}, u_{\{2\}}, u_{\{1\}}+u_{\{2\}})\in \Gamma_1(2) \subset \overline{\Gamma(2)}$.
By Theorem~\ref{thm:cvx}, $\overline{\Gamma(2)}$ is an orthogonally convex cone,
from which it follows that the ray 
\ben
R_{u_{\{1\}}, u_{\{2\}}} =\{(u_{\{1\}}, u_{\{2\}}, u_{\{1,2\}}): u_{\{1,2\}}\geq u_{\{1\}}+ u_{\{2\}} \}
\een
is a subset of $\overline{\Gamma(2)}$. 

We wish to now extend this subset relation to the one remaining circumstance-- 
namely when either $u_{\{1\}}$ or  $u_{\{2\}}$ (or both) is 0.
However this immediately follows because of the closure of $\overline{\Gamma(2)}$.
(Another way to treat this is by extending another example of Bobkov and Chistyakov \cite[Example 1]{BC15:1}.
For every $\epsilon>0$, consider the density 
\ben
f_{\epsilon}(x) = \frac{1}{x(\log \frac{1}{x})^{1+\epsilon}} ,  \quad x\in (0, \frac{1}{e}) .
\een
The entropy of $f_{\epsilon}$ exists for every $\epsilon>0$, and if
$X$ is drawn from $f_{\epsilon}$, then $h(X)> -\infty$ if and only if 
$\epsilon>1$. Furthermore, if $Y$ is drawn from $f_{\epsilon'}$ for some $\epsilon'>0$,
then the density of $X+Y$ behaves near 0 like $f_{\epsilon+\epsilon'}$, so that 
$h(X+Y)> -\infty$ if and only if $\epsilon+\epsilon' >1$.
To construct examples with $u_{\{1\}}=u_{\{2\}}=0$, consider i.i.d. random variables $X, X'$ distributed according to $f_{1}$;
their convolution has finite entropy by the above reasoning. Now simply by scaling, i.e., considering 
$aX$ and $aX'$ for $a>0$, one can generate examples
where $u_{\{1\}}=u_{\{2\}}=0$, and $u_{\{1,2\}}$ is any positive real number.)

We have now proved that $R_{u_{\{1\}}, u_{\{2\}}} \subset \overline{\Gamma(2)}$ 
for arbitrary $u_{\{1\}}, u_{\{2\}} \geq 0$.
Since 
\ben
\Gamma_{FSA}(2) =\cup_{u_{\{1\}}, u_{\{2\}} \geq 0} R_{u_{\{1\}}, u_{\{2\}}} ,
\een
we deduce that 
$\Gamma_{FSA}(2) \subset \overline{\Gamma(2)}$. Since the opposite inclusion follows from
Theorem~\ref{thm:fsa} and the fact that $\Gamma_{FSA}(2)$ is closed, we have 
completed the proof of the fact that $\Gamma_{FSA}(2) = \overline{\Gamma(2)}$.

\subsection{Proof of Theorem~\ref{thm:dim3}}

Fix $a, b, c>0$. Consider the sets
\ben\begin{split}
R_{a, b, c}= &\{u_{\{1,2,3\}}: u=(u_{\{1\}}, u_{\{2\}}, u_{\{3\}}, u_{\{1,2\}}, u_{\{2,3\}}, u_{\{1,3\}}, u_{\{1,2,3\}}) \in \Gamma(3), \\
&u_{\{1\}}=a, u_{\{2\}}=b, u_{\{3\}}=c, \\
& u_{\{1,2\}}=a+b, u_{\{2,3\}}=b+c, u_{\{1,3\}}=c+a \}
\end{split}\een
and
\ben\begin{split}
R'_{a, b, c}= &\{u_{\{1,2,3\}}: u=(u_{\{1\}}, u_{\{2\}}, u_{\{3\}}, u_{\{1,2\}}, u_{\{2,3\}}, u_{\{1,3\}}, u_{\{1,2,3\}}) \in \Gamma_{FSA}(3), \\
&u_{\{1\}}=a, u_{\{2\}}=b, u_{\{3\}}=c, \\
& u_{\{1,2\}}=a+b, u_{\{2,3\}}=b+c, u_{\{1,3\}}=c+a \}
\end{split}\een
Then $R_{a, b, c}$ is the singleton containing the number $a+b+c$; this follows from the equality conditions
for the entropy power inequality, which mandate based on the defining conditions
of $R_{a, b, c}$ that $X_1, X_2$ and $X_3$ are Gaussian
with proportional covariance matrices. On the other hand, 
$R'_{a, b, c}$ is $\{ u_{\{1,2,3\}} \in \Rpl: u_{\{1,2,3\}} \geq a+b+c \}$.
This implies that  $\overline{\Gamma(3)}$ is a strict subset of $\Gamma_{FSA}(3)$.

\subsection{Proof of Theorem~\ref{thm:infStam}}

It is easy to see from the definitions that $\overline{\Gamma(n)}\supset \Gamma_\infty(n)$.

We wish to show now that $\overline{\Gamma(n)}\subset \Gamma_\infty(n)$.
Let $a\in \overline{\Gamma(n)}$. Then there exists a sequence $a_m\in \Gamma_{d_m}(n)$
such that $a_m\ra a$ in $\RL^{2^n-1}$; let us write $a_m(s)=e^{2h(\sum_{i\in s} X_{m,i})/d_m}$, where each $X_{m,i}$
is a $d_m$-dimensional random vector.
We may always assume, without loss of generality, that the sequence $d_m \ra\infty$ as $m\ra\infty$.
This is because $\Gamma_{d}(n)$ is trivially embedded in $\Gamma_{ld}(n)$ for any positive integer $l$;
so we can artificially enforce the condition that $d_m\geq m$ for each $m$ 
in order to make $\infty$ the unique limit point of the sequence $(d_m)$.

For each $k\in\Nat$, 
set $D_k=\prod_{m=1}^k d_m$, and
consider the random vector $Z$ obtained by stacking $D_0=1$ copy of $X_1$, $D_1=d_1$ copies of $X_2$,
$D_2=d_1 d_2$ copies of $X_3$, and so on (to infinity). Let us call $Z_k$ the result of this process
if we stopped at the $k$-th step. Note that $Z_k$ is a random matrix with $n$ columns, 
and $\mathcal{S}_k$ rows of real-valued random variables, where $\mathcal{S}_k:=d_1+D_1 d_2 + \ldots D_{k-1}d_k = \sum_{m=1}^k D_m$
and $D_m\geq m!$ by construction.
We claim that $Z$ is a stochastic process whose point in the infinite-dimensional Stam region is $a$,
which would prove our desired result, and devote the rest of this section to proving this claim.

Denote by $A_k$ the point in the Stam region determined by $Z_k$. Then
\ben
A_k(s)= \prod_{m=1}^k N\bigg(\sum_{i\in s} Z_{m,i}\bigg)^{D_m/\mathcal{S}_k} =  \prod_{m=1}^k a_m(s)^{D_m/\mathcal{S}_k} ,
\een 
so that
\be\label{eq:cesaro}
\log A_k(s)=  \sum_{m=1}^k \frac{D_m}{\mathcal{S}_k} \log a_m(s).
\ee

We now need a basic lemma about limits\footnote{This lemma may be seen as a generalization of the basic theorem
about Ces\`aro means (see, e.g., \cite{Har92:book}). Indeed, while the Ces\`aro mean weights the original sequence
using the weights $c_k=1$, and the original sequence may be seen as a limit of weighted versions of itself (using weights
$c_k=\alpha^k$ with $\alpha\ra\infty$), the lemma treats weights that are in between these two extremes.}.

\begin{lem}\label{lem:cesaro}
Suppose $b_k \ra b$ as $k\ra\infty$, and 
$$
S_n=\frac{\sum_{k=1}^n c_k b_k}{\sum_{k=1}^n c_k} ,
$$
with the positive coefficients $c_k$ satisfying, for some $\alpha>1$,  $c_k\geq \alpha c_{k-1}$ for each $k\in\Nat$.
Then $S_n\ra b$ as $n\ra\infty$.
\end{lem}

\begin{proof}
By considering the sequence $(b_k-b:k\in\Nat)$, it suffices to prove the result for the case where $b=0$;
so we will now assume this without loss of generality. Given that $b_k \ra 0$, we know that for every $\epsilon>0$,
there exists $m(\epsilon)$ such that $|b_k|<\epsilon$ for every $k>m(\epsilon)$. We wish to show that
for every $\epsilon>0$,
there exists $M(\epsilon)$ such that $|S_n|<\epsilon$ for every $n>M(\epsilon)$.
We claim that 
\ben
M(\epsilon):=m(\epsilon/2) + \frac{1}{\log \alpha} \log \bigg[\frac{2}{\epsilon}  m(\epsilon/2) S_{m(\epsilon/2)} \bigg]
\een
serves the purpose.

Indeed, since 
\ben
c_{M(\epsilon)}\geq \alpha^{M(\epsilon)-m(\epsilon/2)} c_{m(\epsilon/2)}= \frac{2}{\epsilon}  m(\epsilon/2) S_{m(\epsilon/2)} c_{m(\epsilon/2)},
\een
we have that
\ben
\sum_{k=1}^{M(\epsilon)} c_k \geq \frac{2}{\epsilon}  S_{m(\epsilon/2)} m(\epsilon/2) c_{m(\epsilon/2)}
\geq  \frac{2}{\epsilon} S_{m(\epsilon/2)} \sum_{k=1}^{m(\epsilon/2)} c_k
= \frac{2}{\epsilon} \bigg| \sum_{k=1}^{m(\epsilon/2)} c_k b_k \bigg| .
\een
Now we can deduce
\ben\begin{split}
|S_{M(\epsilon)}| &\leq \frac{1}{\sum_{k=1}^{M(\epsilon)} c_k}  \bigg| \sum_{k=1}^{m(\epsilon/2)} c_k b_k \bigg|  
+ \frac{1}{\sum_{k=1}^{M(\epsilon)} c_k} \bigg| \sum_{k=m(\epsilon/2)+1}^{M(\epsilon)} c_k b_k \bigg| \\
& \leq \frac{\epsilon}{2} + \frac{1}{\sum_{k=1}^{M(\epsilon)} c_k}  \sum_{k=m(\epsilon/2)+1}^{M(\epsilon)} c_k  \frac{\epsilon}{2} \\
&< \epsilon ,
\end{split}\een
as required.
\end{proof}

Combining Lemma~\ref{lem:cesaro} (with $c_k=D_k, \alpha=2$ and $b_k=\log a_k$) with the continuity of the exponential and logarithmic functions,
it follows from the identity \eqref{eq:cesaro} and the fact that $a_m\ra a$ as $m\ra\infty$ that
$A_k\ra a$ as $k\ra\infty$.
In particular, since the entropy power rate of $\sum_{i\in s} Z_i$ is just the limit of entropy powers
of $\sum_{i\in s} Z_{k,i}$ by definition, we have that the entropy power rate of  $\sum_{i\in s} Z_i$
is $a(s)$, and consequently that the point $a$ is in $\Gamma_\infty(n)$.

\section{Discussion of a question} 
\label{sec:conj}

Let us start by proving a general proposition about determinants of sums of positive-definite matrices.

\begin{prop}\label{prop:supdet}
Let $M_1,M_2,\cdots,M_n$ be $d \times d$ positive-definite matrices. 
Then, the function $v_c:\mathbb{R}_+^n\rightarrow\mathbb{R}_+$ defined by 
\ben
v_c(x) = \text{\emph{det}}\bigg(\sum_{i=1}^n x_iM_i\bigg)
\een
is supermodular, i.e. for any $x,y \in \mathbb{R}_+^n$ it holds that
\begin{equation}\label{Eq:SupermodularityOfDet}
v_c(x) + v_c(y) \leq v_c(x\vee y) + v_c(x \wedge y)
\end{equation}
where $x \vee y$ denotes the componentwise maximum of $x$ and 
$y$ and $x \wedge y$ denotes the componentwise 
minimum of $x$ and $y$.
\end{prop}

\begin{proof}
Inequality~\eqref{Eq:SupermodularityOfDet} trivially holds when $x=0$ or $y=0$. 
To prove that~\eqref{Eq:SupermodularityOfDet} also holds otherwise, it suffices to show by
Proposition~\ref{prop:diff}  (since $v_c$ is smooth)  that
$\frac{\partial ^2 v_c(x)}{\partial x_j \partial x_i}\geq 0$ 
for any $x\in \Rpl^n$. We note that, setting $M=\sum_{i=1}^n x_i M_i$,
\ben
\frac{\partial v_c(x)}{\partial x_i} = \det(M) \tr(M^{-1}M_i)
\een
and
\be\begin{split}
\frac{\partial^2 v_c(x)}{\partial x_j \partial x_i} 
&= \text{det}(M) \big[\tr(M^{-1}M_i) \tr(M^{-1}M_j)- \tr(M^{-1} M_j M^{-1} M_i) \big].
\label{Eq:SecondDerivative}
\end{split}\ee
However, for any two positive definite matrices $A$ and $B$, 
$\text{det}(A) \geq 0$ and 
$\text{trace}(A)\text{trace}(B) \geq \text{trace}(AB)$ (see e.g., Zhang [16], page 166). 
Hence, both terms on the right-hand side of~\eqref{Eq:SecondDerivative} are always non-negative.
\end{proof}
\vspace{.1in}

By Lemma~\ref{lem:set-rn},  Proposition~\ref{prop:supdet} immediately yields:

\begin{thm}\label{thm:supdet}
For any positive-definite matrices $A, B$ and $C$, we have
\be\label{det}
\det(A+B+C) +\det(A) \geq \det(A+B) + \det(A+C) .
\ee
\end{thm}

In other words, the set function $\setS\mapsto e^{2h(\sum_{i\in\setS} X_i)}$
is supermodular if all the random vectors $X_i$ are Gaussian
(since in this case,  the functional $e^{2h}$ is-- up to irrelevant constants--
the determinant of the covariance matrix).

It is natural to wonder how far this phenomenon extends. For example, one might ask
if it holds for arbitrary distributions of the $X_i$, which-- if true-- 
would imply that 
$\Gamma_1(n) \subset \Gamma_{SM}(n)$, thus refining the inclusion
$\Gamma_1(n) \subset \Gamma_{FSA}(n)$ that follows from Theorem~\ref{thm:fsa}.
However, this turns out to be false. Indeed, we will see that even a much more 
restrictive statement is false, but to discuss it, we need to take a detour through some
convex geometry.

It was recently observed in \cite{FMMZ18} (by combining the determinant inequality
of Theorem~\ref{thm:supdet} with some tools from optimal transportation theory)
that the supermodularity property of Theorem~\ref{thm:supdet}  can be extended from determinants to volumes
of convex sets.

\begin{thm}\cite{FMMZ18}\label{thm:vol-sup}
If  $B_i$ are convex sets in $\RL^d$,
\be\label{eq:vol-sup}
\vol_d(B_1+B_2+B_3) + \vol_d(B_1) \geq \vol_d(B_1+B_2) + \vol_d(B_1+B_3)  .
\ee
\end{thm}

Theorem~\ref{thm:vol-sup}  is a generalization of Theorem~\ref{thm:supdet} since
volumes of ellipsoids and parallelotopes (which are all convex) are given
by determinants, and since Minkowski summation of these objects correspond 
to addition of the corresponding positive-definite matrices.
Theorem~\ref{thm:vol-sup} does not extend to  arbitrary compact sets.
However,  \cite{FMMZ18} also poses the question of whether the inequality
\eqref{eq:vol-sup} holds when $B_1$ is compact and convex, while $B_2, B_3$ are arbitrary
compact sets; and it observes that the  answer to this question is affirmative in dimension 1.

As is now well known, the entropy power inequality $\calN(X+Y)\geq \calN(X)+\calN(Y)$
resembles in many ways the Brunn-Minkowski inequality, which is a very important
inequality in mathematics, and a cornerstone of convex geometry in particular. 
This was first noticed by Costa and Cover \cite{CC84}  (see also \cite{DCT91, SV00, WM14} for other 
aspects of this connection). The analogies between inequalities in information theory
and those in convex geometry have been explored quite a bit in recent years 
(see, e.g., the survey \cite{MMX17:0} and references therein), based loosely
on the understanding that the functional $A\mapsto \vol_d(A)^{1/d}$ in the geometry 
of compact subsets of $\R^d$, and the functional $\calL(X)\mapsto \calN(X)$ in 
probability are analogous to each other in many ways. 
In the dictionary that relates notions in convex geometry to those in probability,
the natural analog of a convex set is a log-concave distribution.
Recall that a log-concave distribution is one that has a
density of the form $e^{-V}$, where $V:\R^d\ra \R\cup\{\infty\}$ is a convex function.

Motivated by the afore-mentioned question of \cite{FMMZ18}, it is tempting to consider the following question
which would be its natural probabilistic analog: 

\begin{question}\label{q:sup}
If $X, Y, Z$ are independent $\RL^d$-valued random vectors with $Z$ having a log-concave distribution,
is it true that
\ben
e^{2h(X+Y+Z)} + e^{2h(Z)} \geq e^{2h(X+Z)} + e^{2h(Y+Z)} ?
\een
\end{question}

If the answer to Question~\ref{q:sup} is positive, it would imply that for  independent $\RL^d$-valued random vectors
$X_1,\ldots, X_n$  with log-concave distributions, the set function
\ben
v(\setS)=e^{2h(\sum_{i\in\setS} X_i)}, \quad\setS\subset [n] 
\een
is supermodular.
Unfortunately, however, we will now show that the answer to Question~\ref{q:sup} is negative.

\begin{prop}\label{prop:no}
There exist independent $\RL$-valued random variables $X, Y, Z$, each with log-concave distributions,
such that
\ben
e^{2h(X+Y+Z)} + e^{2h(Z)} < e^{2h(X+Z)} + e^{2h(Y+Z)} ?
\een
In particular, the answer to Question~\ref{q:sup} is negative, already in dimension 1.
\end{prop}
\begin{proof}
Suppose the answer to Question~\ref{q:sup} were yes for $d=1$. Then, given any $\epsilon>0$,
taking $X_\epsilon$ to be Gaussian with variance $\epsilon$ would yield
\ben
\calN(X_\epsilon+Y+Z) + \calN(Z) \geq  \calN(X_\epsilon+Z) + \calN(Y+Z) .
\een
On rearrangement, we have
\ben
\calN(X_\epsilon+Y+Z) - \calN(Y+Z)
\geq \calN(X_\epsilon+Z) - \calN(Z) .
\een
Dividing by $\epsilon$ and taking the limit as  $\epsilon\ra 0$ gives
\be\label{eq:diff-ineq}
\frac{d}{d\epsilon}\bigg|_{\epsilon=0} \calN(X_\epsilon+Y+Z) 
\geq \frac{d}{d\epsilon}\bigg|_{\epsilon=0} \calN(X_\epsilon+Z) .
\ee

The classical de Bruijn identity \cite{Sta59}, which is easily proved by observing
that the density of $X_\epsilon+U$ satisfies the heat equation when 
$X_\epsilon$ is a Gaussian of variance $\epsilon$ independent of $U$,
asserts that
\ben
\frac{d}{d\epsilon}\bigg|_{\epsilon=0} h(X_\epsilon+U) = I(U),
\een
where $I$ is the Fisher information, i.e., if $U$ has density $u$, 
\ben
I(U) =\int_{\RL} \frac{u'(x)^2}{u(x)} dx .
\een
As observed by Stam \cite{Sta59}, the chain rule for differentiation implies that
\be\label{eq:diff-ep}
\frac{d}{d\epsilon}\bigg|_{\epsilon=0} \calN(X_\epsilon+U) = \calN(U) I(U).
\ee

Putting together the inequality \eqref{eq:diff-ineq} and the identity \eqref{eq:diff-ep}, we conclude 
that\footnote{Since $\calN(X+Y)\geq \calN(X)$ and $I(X+Y)\leq I(X)$, if such an inequality were true, 
it would capture how $\calN$ and $I$ compete in their behavior on convolution in the log-concave setting. 
Observe that $\calN(X)I(X)$ is a affine-invariant functional in dimension 1,
and that a limiting argument based on the entropy power inequality  (see, e.g., \cite{Sta59})
yields the entropic isoperimetric inequality
\ben
\calN(X) I(X) \geq 2\pi e ,
\een
with equality if and only if $X$ is Gaussian.}
\ben
\calN(X+Y)I(X+Y) \geq \calN(X) I(X) .
\een
Setting
$T_n=\sum_{i=1}^n X_i$ and $S_n=\frac{1}{\sqrt{n}} T_n$ for i.i.d. 
random variables $X_i$ with the same distribution as $X$,
we must have
\ben
\calN(S_n)I(S_n)= \calN(T_n)I(T_n) \geq \calN(T_{n-1})I(T_{n-1})= \calN(S_{n-1})I(S_{n-1}) ,
\een
which upon taking the limit as $n\ra\infty$ and using the entropic central limit theorem
of Barron \cite{Bar86} and the corresponding statement for Fisher information \cite[Theorem 1.6]{JB04},
yields the inequality $2\pi e \geq \calN(X)I(X)$. Since this contradicts the entropic isoperimetric inequality,
the assumption that we started with-- namely, that the answer to Question~\ref{q:sup} is yes for $d=1$--
cannot hold.
\end{proof}

Since the answer to Question~\ref{q:sup} is negative in dimension 1,
whereas the corresponding question for volumes of sets has a positive
answer in dimension 1 as observed in \cite{FMMZ18}, Proposition~\ref{prop:no}
adds to the examples where the analogy between the probabilistic
and convex geometric statements breaks down.

Proposition~\ref{prop:no} may also be seen as a strengthening of Theorem~\ref{thm:sm}.
While the proof of Theorem~\ref{thm:sm} in Section~\ref{sec:bad-epi} shows that 
$\Gamma_d(n) \cap \Gamma_{SM}(n)^c\neq \phi$ for $d>1$ and $n\geq 3$,
Proposition~\ref{prop:no} shows that for $n\geq 3$, even $\Gamma_1(n)$
contains points outside of the supermodularity cone $\Gamma_{SM}(n)$.
Indeed, by examining the proof, it shows something much stronger:
even if we restrict ourselves to the class of one-dimensional random variables $X, Y, Z$ 
with $Z$ being Gaussian and $X, Y$ being i.i.d. and drawn from a log-concave distribution,
supermodularity of entropy power fails.

\section{Remarks}
\label{sec:rmks}

\begin{enumerate}

\item The classical entropy power inequality 
of Shannon \cite{Sha48} and Stam \cite{Sta59} states that 
if $X_i$ are independent $\RL^d$-valued random vectors,
\be\label{sha-epi}
\calN(X_{1}+\ldots+X_{n}) \geq \sum_{j=1}^{n} \calN(X_{j}) .
\ee
Motivated by the long-standing monotonicity conjecture 
for the entropic central limit theorem, 
Artstein, Ball, Barthe and Naor \cite{ABBN04:1} proved
a new entropy power inequality:
\be\label{abbn-epi}
\calN(X_{1}+\ldots+X_{n}) \geq \frac{1}{n-1} \sum_{i=1}^{n} \calN \bigg( \sum_{j\neq i} X_{j}\bigg) .
\ee
Simplified proofs of inequality \eqref{abbn-epi} were independently given by \cite{MB06:isit, TV06, Shl07}. 
Subsequently, \cite{MB07} proved the inequality 
\eqref{r-epi} based on the maximum degree of the hypergraph $\collS$; 
this contains the inequalities \eqref{abbn-epi} and \eqref{sha-epi}
since for $\collS=\collS_{m}$, namely the collection of all
subsets of indices of size $m$, we have $r=\binom{n-1}{m-1}$.
Theorem~\ref{thm:fsa} says that
for any fractional partition $\beta$ using a
collection $\collS$ of subsets of $[n]$, 
\be\label{eq:fsa}
\calN(X_{1}+\ldots+X_{n}) \geq \sum_{\setS\in\collS} 
\bs \calN\bigg(\sum_{j\in\setS} X_{j}\bigg) .
\ee

Since coefficients that are identical to the reciprocal of the degree
of a regular hypergraph constitute a fractional partition, Theorem~\ref{thm:fsa} subsumes and extends 
all the inequalities discussed above.

\item What are the basic properties of the Stam region? Apart from the theorems 
stated in the introduction, we do not know any other structural properties of
$\Gamma_d(n)$ or $\Gamma(n)$. 
There are many natural questions, such as whether $\Gamma(n)$ or its closure
is convex, that remain to be answered.

\item There are additional constraints on the Stam region for $n\geq 3$ that have not been discussed so far,
but these are nonlinear. Specifically, it was observed in \cite{Mad08:itw} 
(see \cite{MMT10:itw, MMT12, BM11:cras, BM12:jfa, BM13:goetze, KM14, MK10:isit, MK18}
for various applications and generalizations) that the 
differential entropy of sums of independent random vectors is submodular, i.e.,
\ben
h(X_1+X_2+X_3)+h(X_1)\leq h(X_1+X_2)+h(X_1+X_3).
\een
Written in terms of entropy powers, this implies
\be\label{eq:submod}
\calN(X_1 + X_2 +X_3) \calN(X_1)\leq \calN(X_1+X_2) \calN(X_1+X_3),
\ee
which is a nonlinear constraint on the Stam region for any $n\geq 3$.

\item The study of the Stam region is analogous in some sense to the study of the entropic region
defined using the joint entropy of subsets of random variables, on which there has 
been much progress in recent decades.
Let $X_{\setS}$ denote $(X_{i}:i\in\setS)$
and $H$ denote the discrete entropy,
where $(X_{i}:i\in[n])$ is a collection of {\it dependent} random variables
taking values in some finite set. 
Fujishige \cite{Fuj78} observed that $g(\setS)=H(X_{\setS})$
is a submodular set function. Consequent entropy inequalities obtained by Han \cite{Han78}
and Shearer \cite{CGFS86} became influential in information theory and
in combinatorics respectively. These inequalities were unified and generalized 
in \cite{MP82, MT10}, where it was shown that for any submodular set function
$g:2^{[n]}\ra\Rpl$ (and in particular for the joint entropy), one has
\be\label{joint}
g([n])\leq \sumS \bs g(\setS) ,
\ee
as well as corresponding lower bounds that we do not state here,
for any fractional partition $\beta$ using any collection of sets $\collS$.
Motivated by the problem of characterizing the entropic region (or equivalently, the class of
all entropy inequalities for the joint distributions of a collection of 
dependent random variables), Zhang and Yeung \cite{ZY98} observed
in 1998 that there exist so-called non-Shannon 
inequalities that do not follow from submodularity of joint entropy;
these have seen much active investigation recently (see, e.g., 
Mat\'u$\check{\text{s}}$ \cite{Mat07:isit}, who showed the remarkable fact that the entropic
region is not polyhedral if $n\geq 4$).
While Theorem~\ref{thm:fsa} is analogous in some sense to the inequality \eqref{joint}, 
Theorem~\ref{thm:sm} shows that the analogue of Fujishige's submodularity is not true.
In particular, the question of whether there exist entropy inequalities for sums
that are formally analogous to non-Shannon inequalities is ill-posed, since we do not have
supermodularity of entropy power to start with.

\item Theorem~\ref{thm:fsa} is a more informative 
statement than its predecessors such as \eqref{r-epi},
as pointed out in \cite{Mad08:game}. Recall that
entropy power inequalities have been key to the determination of some capacity
and rate regions, and that rate regions for several multi-user
problems (such as the $m$-user Slepian-Wolf problem) 
involve subset sum constraints. Vaguely motivated by this, one
may consider the ``region'' of all $(R_{1},\ldots,R_{n})\in\RL_{+}^{n}$ satisfying
 $\sum_{j\in\setS}R_{j} \geq \calN(T^{\setS})$ for each $\setS\subset[n]$.
Then Theorem~\ref{thm:fsa} is equivalent to the existence of a point in this region such that the total sum
 $\sum_{j\in[n]}R_{j} = \calN(T^{[n]})$.
Although we are not yet aware of a specific multiuser capacity problem
with precisely this rate region, this fact appears intriguing.

\item As discussed in Section~\ref{sec:conj}, there are extensive analogies between inequalities in
convex geometry and entropy inequalities.
The volume analog of the fractional entropy power inequalities proved in this paper, namely
\be\label{eq:fracBM}
\vol_d^{1/d}\left(\sum_{i=1}^n A_i\right) \ge  \sum_{\setS\in\collS} \bs \vol_d^{1/d} \left(\sum_{j\in\setS}A_j\right) 
\ee
for fractional partitions $\beta$, was observed to hold for Minkowski sums of compact convex sets in \cite{BMW11}, and conjectured to hold
more generally for Minkowski sums of compact sets. However, recently it 
was shown in \cite{FMMZ16} that even the weaker inequality
\ben
\vol_d^{1/d}\left(\sum_{i=1}^n A_i\right) \ge \frac{1}{n-1}\sum_{i=1}^n \vol_d^{1/d} \left(\sum_{j\in[n]\setminus\{i\}}A_j\right) ,
\een
corresponding to the leave-one-out subsets, fails for general compact sets 
in dimension 12 or greater.
On the other hand, it was shown in \cite{FMMZ18} that the inequality
\ben
\vol_d\left(\sum_{i=1}^n A_i\right) \ge \frac{1}{n-1}\sum_{i=1}^n \vol_d\left(\sum_{j\in[n]\setminus\{i\}}A_j\right) ,
\een
without the exponents applied to the volumes, holds for general compact sets $A_i$ in $\RL^d$, while
\cite{BMW18} recently showed that the inequality \eqref{eq:fracBM} in fact holds in dimension 1.
\item Shannon (differential) entropy is merely an instance of the more general family of R\'enyi entropies $h_p$ for $p\in [0,\infty]$.
Indeed, the Brunn-Minkowski inequality as well as other inequalities discussed in the preceding remark
may be seen as inequalities for the R\'enyi entropy of order 0, while the Shannon entropy power inequalities 
are inequalities for the R\'enyi entropy of order 1. One can thus seek to understand the R\'enyi
analog of the Stam region-- where, say, one replaces Shannon entropy powers by R\'enyi entropy powers.
However, what we know about R\'enyi entropy power inequalities is so sparse that we are far from being able to address
this question at all meaningfully. Indeed, it is not even clear what the best definition of R\'enyi entropy power is-- 
whether it should be defined as $e^{2h_p(X)/d}$ or with an exponent other than 2, for instance.
The only (at least asymptotically) sharp R\'enyi entropy power inequalities known for $p\neq 0,1$ are for $p=\infty$ and can be found in \cite{MMX17:1}.

\item While this paper focused on characterization of possible inequalities for the entropies
of sums of independent random vectors in $\RL^d$, the same question also makes sense (and is interesting
both as a basic mathematical question and in view of applications to communication theory) for random
variables taking values in any group. A priori, it is natural to first try groups that have particularly simple
structure-- such as finite cyclic groups or the integers. However, even for these seemingly staid examples,
there is little that can currently be said. Indeed, we do not even know a fully satisfactory analogue of the
entropy power inequality on the integers (some partial results are available in \cite{HV03, SDM11:isit, HAT14, WWM14:isit, WM15:isit, MWW17:1, GSS16:isit, MWW17:1, MWW17:2}).
As in the Euclidean setting, such inequalities for integer-valued random variables also
have connections to probabilistic limit theorems (see, e.g., \cite{BJKM10, Yu09:1, JKM13}).
The only discrete groups for which there exist sharp entropy power inequalities are groups of order that is a power of 2;
these were recently proved in \cite{JA14}.

\end{enumerate}

\section*{Acknowledgment}
The first author thanks Young-Han Kim for suggesting that this note, which was largely written in 2009
but set aside, may be of broader interest and for the encouragement to complete and submit it for publication.
Parts of this work were completed while the first author was at the Institute of Mathematics and its Applications (IMA)
in Minneapolis during spring 2015, and he is grateful to the IMA for its hospitality and providing a conducive 
research environment. The authors also thank the Associate Editor and two anonymous reviewers for comments that 
improved the exposition.

\end{document}